\documentclass[11pt]{article}
\usepackage{graphicx}
\usepackage{pgfplots}
\usepackage[colorlinks=true,citecolor=blue]{hyperref}
\usepackage{natbib,extarrows}
\usepackage{graphicx}
\usepackage{amsfonts}
\usepackage{amsmath}
\usepackage{amssymb}
\usepackage{url}
\usepackage{fancyhdr}
\usepackage{indentfirst}
\usepackage{enumerate}
\usepackage{titlesec}
\usepackage{amsthm}
\usepackage{subcaption}
\usepackage{dsfont}
\usepackage[misc]{ifsym}
\usepackage{soul}

\theoremstyle{definition}
\newtheorem{example}{Example}

\newtheorem{definition}{Definition}

\theoremstyle{plain}
\newtheorem{theorem}{Theorem}
\newtheorem{lemma}{Lemma}
\newtheorem{proposition}{Proposition}

\theoremstyle{remark}
\newtheorem{remark}{Remark}
\allowdisplaybreaks

\usepackage{cases}

\theoremstyle{definition}
         
\def\oo{\infty}                   \def\d{\,\mathrm{d}}
\def\lm{\lambda}                  \def\th{\theta}

\DeclareMathOperator{\sign}{sgn}

\def\N{\mathbb{N}}
\def\P{\mathbb{P}}
\def\p{\mathbb{P}}
\def\E{\mathbb{E}}

\def\R{\mathbb{R}}

\def\H{\mathcal{H}}

\def\d{\mathrm{d}}

\DeclareMathOperator*{\essinf}{ess\text{-}inf}

\def\oo{\infty}

\renewcommand{\(}{\left(}
\renewcommand{\)}{\right)}
\renewcommand{\[}{\left[}
\renewcommand{\]}{\right]}

\usepackage[onehalfspacing]{setspace}

\usepackage{bm}
\usepackage{tikz-qtree}
\usepackage{tikz}

\usepackage{todonotes}

\def\id{\mathds{1}}

\begin{document}

\title{Stochastic dominance for linear combinations of infinite-mean risks}

\author{Yuyu Chen\thanks{Department of Economics, University of Melbourne,  Australia. \Letter~{\scriptsize\url{yuyu.chen@unimelb.edu.au}}}
\and Taizhong Hu\thanks{Department of Statistics and Finance, University of Science and Technology of China,  China. \Letter~{\scriptsize\url{thu@ustc.edu.cn}}}
\and Seva Shneer \thanks{Department of Actuarial Mathematics and Statistics, Heriot-Watt University,  UK. \Letter~{\scriptsize\url{V.Shneer@hw.ac.uk}}}
\and Zhenfeng Zou\thanks{School of Public Policy and Management, University of Science and Technology of China,  China. \Letter~{\scriptsize\url{zfzou@ustc.edu.cn}}}
}
\date{\today}
\maketitle

\begin{abstract}
In this paper, we establish a sufficient condition for comparing linear combinations of independent and identically distributed infinite-mean random variables under usual stochastic order. 
We introduce a new class of distributions that includes many commonly used heavy-tailed models and show that within this class, a linear combination of random variables is stochastically larger when its weight vector is smaller in the sense of majorization order. We proceed to study the case where each random variable is a compound Poisson sum and demonstrate that if the stochastic dominance relation holds, 
 the summand of the compound Poisson sum belongs to our new class of distributions.
Additional discussions are presented for stable distributions. 

\textbf{Keywords}: infinite mean; heavy-tailed distributions; usual stochastic order; majorization order.

\textbf{MSC2020}: 60E05; 60E15. 

\end{abstract}

\section{Introduction}

 Linear combinations of random variables play a crucial role in various fields, including statistics, operations research, reliability theory, and actuarial science. 
The comparison of such linear combinations, particularly in terms of their variability or riskiness, has been a subject of extensive study. Let $X_1, \dots, X_n$ be independent and identically distributed (iid) random variables. 
Given two non-negative real vectors $(\theta_1, \dots, \theta_n)$ and $(\eta_1, \dots, \eta_n)$ such that $(\theta_1, \dots, \theta_n)$ is smaller than $(\eta_1, \dots, \eta_n)$ in majorization order (see Definition \ref{maj}), a typical problem in this area is to compare $\sum_{i=1}^n \th_i X_i$ and $\sum_{i=1}^n \eta_i X_i$ via some notion of stochastic dominance. 
Generally speaking, a smaller vector in majorization order is more evenly distributed across its coordinates.

An important early contribution is \cite{Pro65}, who showed that for iid symmetric log-concave random variables $X_1, \dots, X_n$, $\sum_{i=1}^n \th_i X_i$ is more peaked than $\sum_{i=1}^n \eta_i X_i$, i.e., $\p(|\sum_{i=1}^n \th_i X_i|\le x)\ge \p(|\sum_{i=1}^n \eta_i X_i|\le x)$ for all $x>0$. That is, a more equal distribution of weights leads to a less variable linear combination. This remarkable result was subsequently extended and found applications in economics, statistics, and operations research; see, e.g., \cite{ibragimov2005new} and the references therein.
For instance, \cite{Ma98} extended Proschan's result to independent but not necessarily identically distributed random variables.
Subsequently, \cite{Ma00} compared linear combinations of random variables using decreasing convex order and  Laplace transform order; see \cite{ZCN11}, \cite{XH11}, and \cite{Mao13} for studies with other notions of stochastic orders.

Of particular note is the extension of \cite{Pro65} established in \cite{ibragimov2005new}, who showed that the peakedness property for log-concave distributions continues to hold for stable distributions with finite mean.  A closely related result is that for all iid finite-mean $X_1,\dots,X_n$, $\sum_{i=1}^n \th_i X_i$ is smaller than $\sum_{i=1}^n \eta_i X_i$ in convex order  (see, e.g., \cite{SS07}, Theorem 3.A.35; \cite{DDGK05}, Property 3.4.48); a random variable $X$ is said to be smaller than a random variable $Y$ in the convex order, if $\E [\phi(X)]\le \E [\phi(Y)]$ for all convex functions $\phi:\R\to\R$, provided that both expectations exist. This implies $\mathrm{Var}(\sum_{i=1}^n \th_i X_i)\le \mathrm{Var}(\sum_{i=1}^n \eta_i X_i)$ provided that $X_1,\dots,X_n$ have a finite second moment. The above peakedness property and the relation in convex order both suggest that $\sum_{i=1}^n \th_i X_i$ is less spread-out than $\sum_{i=1}^n \eta_i X_i$.

One key assumption for the above results to hold is that random variables have a finite mean. If the mean of random variables is infinite, the observations can be quite different or even flipped; see \cite{CW25} for a survey on infinite-mean models in the realms of finance and insurance. For instance, if $X_1,\dots, X_n$ are iid stable random variables with a tail parameter strictly less than 1 (thus infinite mean),  \cite{ibragimov2005new} showed that $\sum_{i=1}^n \th_i X_i$ is less peaked than $\sum_{i=1}^n \eta_i X_i$;  hence a more balanced weight vector leads to a more dispersed linear combination. 
If the stable random variables are positive almost surely, this result can be equivalently written as 
\begin{equation}\label{eq:intro}
    \sum_{i=1}^n \eta_i X_i \leq_{\rm st} \sum_{i=1}^n \th_i X_i~ {\rm for} ~ (\theta_1, \dots, \theta_n) \preceq (\eta_1, \dots, \eta_n),
\end{equation}
where $\leq_{\rm st}$ denotes  usual stochastic order and $\preceq$ denotes majorization order. For two random variables $X$ and $Y$, $X$ is said to be smaller than $Y$ in \emph{usual stochastic order} if $\p(X\le x)\ge \p(Y\le x)$ for all $x\in\R$; for properties of stochastic orders, we refer to \cite{MS02} and \cite{SS07}.

In recent years, a growing body of research has focused on  \eqref{eq:intro} and related stochastic dominance relations for extremely heavy-tailed random variables, i.e.,  the means of the random variables are infinite or do not exist. 
\cite{CHWZ25}  extended \eqref{eq:intro}  to extremely heavy-tailed Pareto random variables that are not necessarily identically distributed; the case of two iid Pareto random variables with tail parameter 1/2 was previously studied by \cite{embrechts2002correlation}. Several studies considered an important case of \eqref{eq:intro}: \begin{equation}\label{eq1}
    \(\sum_{i=1}^n \th_i\) X_1 \leq_{\rm st} \sum_{i=1}^n \th_i X_i,
\end{equation}
where $(\theta_1, \dots, \theta_n)$ is a nonnegative vector. \cite{CEW22} established \eqref{eq1} for weakly negatively associated super-Pareto random variables. 
Later, \cite{arab2024convex}, \cite{muller2024some}, and \cite{CS24} provided other sufficient conditions of $X_1,\dots,X_n$ for which \eqref{eq1} holds.

Despite these advances, existing proofs of \eqref{eq:intro} rely heavily on the specific mathematical structures of the distributions considered. 
The arguments in \cite{ibragimov2005new} rely on the fact that linear combinations of iid stable random variables remain stable, while those in \cite{CHWZ25} are tailored to the Pareto family. 
Consequently, a unifying framework that encompasses a wide range of extremely heavy-tailed distributions has remained elusive.
This paper fills the gap by providing a general sufficient condition for the stochastic dominance relation \eqref{eq:intro}. Moreover, we provide, for the first time, necessary conditions within a particular class of distributions.

Inequality \eqref{eq:intro} provides an intuitive implication in portfolio diversification: A more diversified portfolio of iid infinite-mean risks is stochastically larger. An early observation of the negative effect of diversification dates back to \citet[p.~271]{FM72}, who showed that diversification can make portfolios more spread-out for infinite-mean risks. The superadditivity of the regulatory risk measure  Value-at-Risk (VaR) was noted by \cite{NEC06} in the context of operational risks, which can have infinite mean. \cite{ibragimov2009portfolio} showed that VaR is superadditive for infinite-mean stable distributions at quantile levels greater than 0.5. The economic analysis in \cite{IJW11} suggested that diversification is suboptimal for infinite-mean risks from the perspectives of both financial intermediaries and society.
An application of \eqref{eq:intro} to the optimal bundling problems can be found in \cite{IW10}.

 We summarize our contributions below.
\begin{itemize}
\item With majorization theory, we show that \eqref{eq:intro} holds for a new class of distributions, referred to as inverted concave distributions (Theorem \ref{thm:main}). This class includes many commonly used heavy-tailed distributions, such as the Pareto distributions and the Fr\'echet distributions, all with infinite mean. Together with some simple closure properties of the inverted concave distributions (Propositions \ref{prop:Hproperty}-\ref {prop:mixture}),  we demonstrate that the new class is rather rich. As mentioned before, the techniques used in \cite{ibragimov2005new} and \cite{CHWZ25} are designed specifically for stable distributions and Pareto distributions, respectively, and there does not seem to be a way to generalize them to other distributions. We develop and apply different methods to prove our main result in Theorem \ref{thm:main}.
\item Assuming that each of $X_1,\dots,X_n$ in  \eqref{eq:intro} follows a compound Poisson distribution, which is a standard model for portfolio claims and aggregate risks in banking and insurance sectors \citep[see, e.g.,][]{kaas2008modern}, we show that \eqref{eq:intro} holds if and only if the distribution of each summand is inverted concave. This, to the best of our knowledge, is the first necessary result for \eqref{eq:intro} to hold.
\end{itemize}

The rest of this paper is organized as follows. Section \ref{sec:2} introduces the class of inverted concave distributions and establishes \eqref{eq:intro} for this class. Section \ref{sec:3} studies \eqref{eq:intro} for the compound Poisson distribution. Section \ref{sec:4} presents further discussions of \eqref{eq:intro} for stable distributions, complementing the findings of \cite{ibragimov2005new}, who focused on the peakedness property of linear combinations of stable random variables. Section \ref{sec:5} concludes the paper with some open questions. The appendices provide more details for Table \ref{t1}, which contains further examples of inverted concave distributions, as well as a simulation procedure for stable distributions used in the numerical illustration.

\subsection*{Notation}
We fix some notation here. Denote by $\N$ the set of all positive integers, $\R_+$ the set of all non-negative real numbers, and $\R_{++}$ the set of all positive real numbers. Let $\N_0=\N\cup\{0\}$.
For $n\in\N$, let $[n]=\{1,\dots,n\}$. Throughout, $n$ is assumed to be larger than or equal to 2. 
We use $\overline \theta$ and $\overline \eta$ to denote vectors $(\theta_1,\dots,\theta_n)\in\R_+^n$ and $(\eta_1,\dots,\eta_n)\in\R_+^n$, respectively. 
 Denote by $ \varphi_X $ the characteristic function of a random variable $ X$. For a distribution function $ F $, denote its tail function by $ \overline{F} = 1 - F $ and its essential infimum by $ \essinf F $.
We use $\simeq_{\rm st}$ for equality in distribution.

\section{Stochastic dominance and a class of distributions }\label{sec:2}
\subsection{A stochastic dominance relation}
We first recall the notion of majorization order. For a vector $(x_1,\dots,x_n)\in\R^n$, denote by $x_{(i)}$ the $i$th smallest order statistic of $(x_1,\dots,x_n)$.
\begin{definition}\label{maj}
For two vectors $ \overline{\theta} := \(\theta_1,\dots,\theta_n\)$ and $ \overline{\eta} :=  \(\eta_1,\dots,\eta_n\)$ in $\R^n$, we say
\begin{enumerate}[(i)]
\item $ \overline{\theta}$ is dominated by $ \overline{\eta}$ in \emph{weak supermajorization order}, denoted by $ \overline{\theta} \preceq^w  \overline{\eta}$, if $\sum^k_{i=1} \theta_{(i)} \ge \sum^k_{i=1} \eta_{(i)}$  for all $k\in [n]$;
\item $ \overline{\theta}$ is dominated by $ \overline{\eta}$ in \emph{majorization order}, denoted by $ \overline{\theta} \preceq  \overline{\eta}$, if $ \overline{\theta} \preceq^w  \overline{\eta}$ and $\sum^n_{i=1}\theta_i =\sum^n_{i=1}\eta_i$.
\end{enumerate}
\end{definition}

The main focus of the paper is on studying the following stochastic dominance relation

\begin{equation} \label{eq:SD}
\sum_{i=1}^n\eta_i X_i  \le_{\rm st} \sum_{i=1}^n\theta_i X_i\mbox{  for all $\overline \theta,\overline \eta\in\R^n_+$ such that $\overline \theta \preceq \overline \eta$,} \tag{\textcolor{red}{SD}}
\end{equation}
where $ X_1, \dots, X_n$ are iid random variables. Let $X$ be a copy of $ X_1, \dots, X_n$. If $X_1, \dots, X_n\sim F$ satisfy \eqref{eq:SD}, we also say $X$ or $F$ satisfies \eqref{eq:SD}. 
For some $x\in\R$, let
$$
g_x(\theta_1,\dots,\theta_n) = \p\(\sum_{i=1}^n \theta_i X_i > x\),~~\overline \theta\in\R^n_+.
$$ 
Then  \eqref{eq:SD} holds if and only if 
$g_x$ is Schur-concave on  $\mathbb{R}^n_+$ for all $x \in \R$; recall that a function $f$ on $\R^n$ is \emph{Schur-concave} if  $f(\overline \theta)\ge f(\overline \eta) $  whenever $\overline \theta \preceq  \overline \eta$.

We present below some elementary properties of \eqref{eq:SD}.

\begin{proposition}\label{prop:trivial}
The following statements hold.
\begin{enumerate}[{\rm (i)}]
\item If random variable $X$ satisfies \eqref{eq:SD}, then either $X$ is a constant, or $\E[|X|]$ is infinite.
\item If  two independent random variables $X$ and $Y$ satisfy \eqref{eq:SD}, then $X+Y$ satisfies \eqref{eq:SD}.
\item If  random variable $X$  satisfies \eqref{eq:SD}, then $(X+a)/b$ satisfies \eqref{eq:SD} for all $a\in\R$ and $b \in \R_{++}$.
\item If random variable $Y_k$ with any  $k \in \N$ satisfies \eqref{eq:SD} and $Y_k\xrightarrow{\d}X$, then $X$ satisfies \eqref{eq:SD}.
\end{enumerate}
\end{proposition}
\begin{proof}
If $X$ is a constant, \eqref{eq:SD} clearly holds. For a non-degenerate random variable $X$, as
 \eqref{eq:SD} implies  $(\sum_{i=1}^n\theta_i) X  \le_{\rm st} \sum_{i=1}^n\theta_i X_i$, by Theorem 2.3 (b) of \cite{muller2024some}, $\E[|X|]=\infty$.
Statement (ii) is because usual stochastic order is preserved under convolution  \citep[see, e.g.,][Theorem 1.A.3 (b)]{SS07}. Statement (iii) is trivial. Statement (iv)  holds as usual stochastic order is preserved under convergence in distribution  \citep[][Theorem 1.2.14]{MS02}.
\end{proof}

\begin{remark}
Property \eqref{eq:SD} implies that 
$
(X_1+\dots+X_n)/n
$
is increasing in usual stochastic order as $n$ increases, since  $(1/n,\dots,1/n,0)$ dominates  $(1/(n+1),\dots,1/(n+1))$ in majorization order for all $n$. By Proposition \ref{prop:trivial}, if $X_1$ satisfies \eqref{eq:SD} and  $X_1$ is non-degenerate and bounded from below, we have $\E[X_1]=\infty$. Then by the law of large numbers,  $(X_1+\dots+X_n)/n\rightarrow \infty$ almost surely as $n$ goes to infinity  \citep[see, e.g.,][Theorem 2.4.5]{D19}.
\end{remark}

\subsection{Inverted concave distributions}

Next, we introduce a class of distributions for non-negative random variables, namely the inverted concave distribution. We will show later that this rather large class of distributions satisfies property \eqref{eq:SD} and thus, by Proposition \ref{prop:trivial}, the inverted concave distributions have infinite mean.

\begin{definition}
Let $F$ be a distribution function with $\essinf F=0$. We say $F$ is \emph{inverted concave}, denoted by  $F\in \H$, if  $\overline F(1/x)$ is concave in $x\in \R_{++}$. For $X\sim F$ where $F\in\H$, we also write $X\in\H$.
\end{definition}

For a continuous random variable $X$, its \emph{inverted distribution} is defined as  $\p(1/X \leq x)$, $x\in\R$. If $X$ is inverted concave, for $x>0$, $\p(1/X\le x)=\p(X> 1/x)$ is concave, i.e., the inverted distribution
of $X$ is concave, and thus the name. The following result for the inverted concave class will be useful later.
\begin{proposition}\label{prop:Hcont}
If $F\in\H$, then  
\begin{equation*} 
\phi(x_1,x_2)= \overline{F} \left(\frac{1}{x_1}\right)+\overline{F} \left(\frac{1}{x_2}\right),~~(x_1,x_2)\in \R_{++}^2
\end{equation*}
is Schur-concave. Conversely, if $F$ is continuous and $\phi$ is Schur-concave, $F\in \H$.
In addition, if $F$ has density $f$, then $F\in\H$ if and only if $x^2f(x)$ is increasing in $x\in\R_{++}$.
\end{proposition}
\begin{proof}

By  Section 3.C.1 of \cite{MOA11}, if $g:\R_+\rightarrow \R_+$ is concave, then   $\phi(x_1,\dots,x_n)=\sum_{i=1}^n g(x_i)$ is Schur-concave in $(x_1,\dots,x_n)\in\R_+^n$.
By Section  3.C.1.c of \cite{MOA11}, for a continuous function $g$ on an interval $I$, if    $\psi(x_1,x_2)=g(x_1)+g(x_2)$ is Schur-concave on $I^2$, then $g$ is concave.
\end{proof}

\begin{remark}\label{rem:RV}
If distributions have densities of the form
$
f(x) = x^{-\alpha-1} l(x),
$
where the function $l$ is increasing and $\alpha \in (0,1]$, the distributions are inverted concave by Proposition \ref{prop:Hcont}. In addition, if $l$ is a slowly-varying function, by the Karamata representation  \citep[see, e.g.,][Theorem B.1.6]{DF06}, it may be written as
$$
l(x) = \exp\left(c(x) + \int_0^x \frac{\varepsilon(t)}{t} \d t\right),
$$
where $c(x) \to c$ as $x \to \infty$ and $\varepsilon(x) \to 0$ as $x \to \infty$. If $c$ is increasing and $\varepsilon$ is non-negative, the distributions are inverted concave.
\end{remark}

One can use Proposition \ref{prop:Hcont} together with Remark \ref{rem:RV} to show that many commonly used heavy-tailed distributions are inverted concave. We present some examples below. Additional examples are collected in Table \ref{t1}; see Appendix \ref{sec:SHex} for more details of Table \ref{t1}. Some closure properties of inverted concave distributions are collected in Section \ref{sec:closure}.

\begin{example}[Fr\'echet distribution]
 The Fr\'echet distribution function is $F(x)=\exp\(-x^{-\alpha}\)$, $x \in \R_{++}$. As its density function is $f(x)=\alpha x^{-\alpha-1}\exp\(-x^{-\alpha}\)$ for $x \in \R_{++}$, by Remark \ref{rem:RV}, $F\in\H$ if $\alpha\le 1$.
\end{example}
\begin{example}[Absolute Cauchy distribution]\label{ex:absCauchy}
The absolute Cauchy distribution is defined as
$F(x)=(2/\pi)\arctan(x)$ for  $x \in \R_{++}$. As
the density function is $f(x)=2/(\pi(x^2+1))=x^{-2}(2x^2/(\pi(x^2+1)))$ for $x \in \R_{++}$, by Remark \ref{rem:RV},  we have $F \in \H$. \end{example}
\begin{example}[Inverse Gamma distribution]\label{ex:IG}
The density function of the inverse Gamma distribution with parameters $\alpha,\beta \in \R_{++}$ is $f(x)= \beta^\alpha x^{-\alpha-1}\exp\(-\beta/x\)/\Gamma(\alpha)$, $x \in \R_{++}$. By Remark \ref{rem:RV}, if $\alpha\le 1$ and $\beta \in \R_{++}$, we have $F \in \H$. The inverse Gamma distribution with $\alpha=0.5$ is the L{\'e}vy distribution, which also belongs to the class of stable distributions.
\end{example}
\begin{example}[Feller-Pareto distribution]
The Feller-Pareto distribution with parameters $\mu\in\R$ and $\sigma,\gamma,\gamma_1, \gamma_2 \in \R_{++}$, denoted by ${\rm FP}(\mu,\sigma,\gamma,\gamma_1, \gamma_2)$, has the density function
$$f(x)=\frac{1}{\gamma\sigma B(\gamma_1,\gamma_2)}\(\frac{x-\mu}{\sigma}\)^{(\gamma_2/\gamma)-1}\(1+\(\frac{x-\mu}{\sigma}\)^{1/\gamma}\)^{-\gamma_1-\gamma_2},~~x>\mu,$$
where
$B$ is the beta function. The class of  Feller-Pareto distributions is very general. If $\gamma_2=1$, the Feller-Pareto distribution becomes the Pareto type-IV distribution (${\rm FP}(\mu,\sigma, \gamma,\alpha,1)$), which includes the Pareto type-III distribution (${\rm FP}(\mu,\sigma,\gamma,1,1)$), the Pareto type-II distribution (${\rm FP}(\mu,\sigma,1,\alpha,1)$), and the Pareto type-I distribution (${\rm FP} (\sigma, \sigma,1,\alpha,1)$). Note that the Pareto type-IV distribution is essentially the Burr distribution.  See \cite{A15} for the statistical properties of the Feller-Pareto distribution. Assume that $\mu=0$ and $\sigma=1$ and let $h(x)=x^2f(x)$ for $x\in\R_+$. Taking the derivative of $h$, we have 
$$h'(x)=\frac{x^{\gamma_2/\gamma}\(\(x^{1/\gamma}+1\)^{-\gamma_1-\gamma_2-1}\((\gamma-\gamma_1)x^{1/\gamma}+\gamma_2 +\gamma\)\)}{\gamma^2 B(\gamma_1,\gamma_2)}.$$
If $\gamma\ge \gamma_1$, i.e., the Feller-Pareto distribution has infinite mean, $h$ is increasing. By Proposition \ref{prop:Hcont}, the Feller-Pareto distribution is inverted concave if $\gamma\ge \gamma_1$.
\end{example}

\begin{table}[h!] 
{
\begin{center}
\caption{Further examples of inverted concave distributions.}
\begin{tabular}{c |c | c}
\hline
& Distribution functions & Parameters  \\ 
\hline
Pareto distribution  & $F(x)=1-(x+1)^{-\alpha},~~ x \in \R_{++}$ & $\alpha\le 1$    \\
\hline
Log-Pareto distribution & $F(x)=1-(\log(x+1)+1)^{-\alpha},~~x \in \R_{++}$ &$\alpha\le 2$\\
\hline
Inverse Burr distribution & $F(x)=(x^\tau/(x^\tau+1))^\alpha,~~ x \in \R_{++}$ & $\tau\le 1$, $\alpha>0$\\
\hline
Stoppa distribution & $F(x)=(1-(x+1)^{-\alpha})^\beta, ~~x \in \R_{++}$ &$\alpha\le 1$, $\beta\ge 1$\\
\hline
Log-Cauchy distribution & $F(x)=\arctan(\log(x))/\pi+1/2,~~x \in \R_{++}$ &NA\\
\hline
\end{tabular}
\label{t1}
\end{center}
}
\end{table}

\subsection{Main result}

Let $\overline {\theta}$ and $\overline {\eta}$ be two vectors in $\R^n$.  Then $\overline {\theta}$ is a \emph{$T$-transform} of $\overline {\eta}$ if for some $\lambda \in [0,1]$ and $i,j\in [n]$ with $i \neq j$,
\begin{equation*}\label{eq:t-tran}
\theta_i=\lambda \eta_i + (1-\lambda)\eta_j,~\theta_j= (1-\lambda)\eta_i+\lambda \eta_j,
\end{equation*}
and $\theta_k=\eta_k$ for $k\in [n]\setminus \{i,j\}$. The following lemma will be useful in our proof.

\begin{lemma}
\label{lem:maj}
{\rm \citep[][Section 1.A.3]{MOA11}}\ \
For $\overline \theta, \overline \eta\in \R^n$, $\overline \th\preceq \overline \eta$ if and only if $\overline \theta$ can be obtained from $\overline\eta$ by a finite number of $T$-transforms.
\end{lemma}

\begin{theorem}\label{thm:main}
If $F\in \H$, then $F$ satisfies \eqref{eq:SD}.
\end{theorem}
\begin{proof}
We first show that it is sufficient to prove the theorem for the case of two random variables. 
By Lemma \ref{lem:maj}, it suffices to assume that $\overline \theta$ is obtained from $\overline \eta$ by a $T$-transform, i.e., only two components in $\overline \theta$ and $\overline \eta$ are different. Write $\overline \theta=(\lambda \eta_1 + (1-\lambda) \eta_2,(1-\lambda) \eta_1 + \lambda \eta_2, \eta_3,\dots,\eta_n)$ for some $\lambda\in[0,1]$. If the statement of the theorem is correct for two random variables, then
$$
\theta_1 X_1 + \theta_2 X_2=(\lambda \eta_1 + (1-\lambda) \eta_2) X_1 + ((1-\lambda) \eta_1 + \lambda \eta_2) X_2 \ge_{\rm st} \eta_1 X_1 + \eta_2 X_2. 
$$
Then, as $X_1,\dots,X_n$ are independent and  usual stochastic order is closed under convolution  \citep[see, e.g.,][Theorem 1.A.3 (b)]{SS07},
\begin{align*}
\sum_{i=1}^n \theta_i X_i&=(\lambda \eta_1 + (1-\lambda) \eta_2) X_1 + ((1-\lambda) \eta_1 + \lambda \eta_2) X_2 + \sum_{i=3}^n \eta_i X_i\\
& \ge_{\rm st} \eta_1 X_1 + \eta_2 X_2 + \sum_{i=3}^n \eta_i X_i=\sum_{i=1}^n \eta_i X_i.
\end{align*}

It therefore remains to prove the theorem in the case $n=2$. It is clear that we can restrict ourselves to the case $\theta_1+\theta_2=1$. It is also clear that, due to symmetries, it is sufficient to show that
\begin{equation} \label{eq:aux_1}
\p(\theta X_1 + (1-\theta) X_2 > x) \ge \p(\eta X_1 + (1-\eta) X_2 > x)
\end{equation}
holds for all $0 \le \eta \le \theta \le 1/2$ and for all $x\in\R$. To this end, for $a \in (0,1)$, write
\begin{align}\label{eq:remark}
 & \p(a X_1 + (1-a) X_2 > x) \nonumber\\
& = \p(X_1>x,X_2>x) + \p(X_1\le x, a X_1 + (1-a) X_2 > x) + \p(X_2\le x, a X_1 + (1-a) X_2 > x) \nonumber\\
& = \p(X_1>x,X_2>x) + \int_{y=0}^x  \overline{F}\left(\frac{x-a y}{1-a}\right)\d F(y) + \int_{y=0}^x  \overline{F}\left(\frac{x-(1-a) y}{a}\right)\d F(y).
\end{align}
From the above, it is now clear that, in order to show \eqref{eq:aux_1} holds for all $0 < \eta \le \theta \le 1/2$ and for all $x\in\R$, it is sufficient to show that
\begin{equation*} \label{eq:aux_2}
\overline{F}\left(\frac{x-\theta y}{1-\theta}\right) + \overline{F}\left(\frac{x-(1-\theta) y}{\theta}\right) \ge
\overline{F}\left(\frac{x-\eta y}{1-\eta}\right) + \overline{F}\left(\frac{x-(1-\eta) y}{\eta}\right)
\end{equation*}
for all $y \le x$ and for all $0 < \eta \le \theta \le 1/2$. This can be rewritten as
\begin{equation}\label{eq}
\overline{F}\left(y + \frac{x-y}{1-\theta}\right) + \overline{F}\left(y+\frac{x-y}{\theta}\right) \ge
\overline{F}\left(y+\frac{x-y}{1-\eta}\right) + \overline{F}\left(y+\frac{x-y}{\eta}\right)
\end{equation}
and, with conditions above,  \eqref{eq} can also be stated as
\begin{equation} \label{eq:suff_1}
\overline{F}\left(y + \frac{z}{1-\theta}\right) + \overline{F}\left(y+\frac{z}{\theta}\right) \ge
\overline{F}\left(y+\frac{z}{1-\eta}\right) + \overline{F}\left(y+\frac{z}{\eta}\right)
\end{equation}
for all $y, z \in \R_+$ and for all $0 < \eta \le \theta \le 1/2$. The above holds if $z=0$ and thus we only need to consider the case $ z\in \R_{++}$.
 Let $h(x)=\overline F(1/x)$ for $x\in\R_{++}$ and $l(w)=1/(y+z/w)$ for $w\in \R_{++}$, where $ y\in \R_{+}$ and $ z\in \R_{++}$.
Since $F\in\H$, $h$ is concave. It is easy to check that $l$ is also concave. By Lemma \ref{lem:maj}, there exists some $\lambda\in[0,1]$ such that $\theta=(1-\lambda) (1-\eta) + \lambda \eta$ and  $1-\theta=\lambda (1-\eta) + (1-\lambda) \eta$. Therefore, we have 
\begin{align*}
\overline{F}\left(y + \frac{z}{1-\theta}\right) + \overline{F}\left(y+\frac{z}{\theta}\right) &= h(l(1-\theta))+h(l(\theta))\\
&=h(l(\lambda (1-\eta) + (1-\lambda) \eta))+h(l((1-\lambda) (1-\eta) + \lambda \eta))\\
&\ge h(\lambda l(1-\eta)+ (1-\lambda) l (\eta))+h((1-\lambda) l (1-\eta)+ \lambda l(\eta))\\
&\ge h(l(1-\eta))+h(l(\eta)) =
\overline{F}\left(y+\frac{z}{1-\eta}\right) + \overline{F}\left(y+\frac{z}{\eta}\right).
\end{align*}
Hence, \eqref{eq:aux_1} holds for $0 < \eta \le \theta \le 1/2$. Moreover, for any $a\in(0,1)$, we have 
\begin{align}\label{eq:proofSD*}
& \p(a X_1 + (1-a) X_2 > x)   \nonumber\\
& = \p(X_1>x,X_2>x) + \int_{y=0}^x  \overline{F}\left(\frac{x-a y}{1-a}\right)\d F(y) + \int_{y=0}^x  \overline{F}\left(\frac{x-(1-a) y}{a}\right)\d F(y)  \nonumber\\
& \ge \p(X_1 > x, X_2 > x) + \int_{y=0}^x  \overline{F}\left(\frac{x}{1-a}\right)\d F(y) + \int_{y=0}^x  \overline{F}\left(\frac{x}{a}\right)\d F(y) \nonumber\\
&\ge \p(X_1>x,X_2>x) + \int_{y=0}^x  \overline{F}\left(x\right)\d F(y)  \\
&=\p(X_1>x,X_2>x)+ \p(X_1>x,X_2\le x)= \p(X_1>x). \nonumber 
\end{align}
Inequality \eqref{eq:proofSD*} is due to Proposition \ref{prop:Hcont} and right-continuity of $\overline F$. Hence, \eqref{eq:aux_1} holds for $\eta=0$. 
The proof is complete.
\end{proof}

\begin{remark}
The proof of Theorem \ref{thm:main} is significantly 
different from \cite{ibragimov2005new} and \cite{CHWZ25} who established \eqref{eq:SD} for some infinite-mean stable distributions and all infinite-mean Pareto distributions, respectively. \cite{ibragimov2005new} mainly used the fact that a linear combination of iid stable random variables follows the stable distribution; see Section \ref{sec:4}. The techniques in \cite{CHWZ25} depend on the assumption of Pareto distributions and have been used to derive a more general version of \eqref{eq:SD} for non-identically distributed Pareto random variables. On the other hand, our proof utilizes the symmetry property of functions and thus does not seem to work for non-identical distributions.
\end{remark}

\begin{remark}\label{rem:unbounded}
We note that \eqref{eq:SD} can also hold for random variables that are not bounded from below. A classic example is the standard Cauchy distribution, with the distribution function given by 
$$F(x)=\frac{1}{\pi}\arctan(x)+\frac{1}{2},~~x\in\R.$$
It is well known that for iid standard Cauchy random variables $X_1,\dots,X_n$ and for any $\overline \theta\in\R_+^n$, $\sum_{i=1}^n\theta_iX_i\simeq_{\rm st} (\sum_{i=1}^n\theta_i)X_1$ and thus \eqref{eq:SD} holds. A question is whether the proof of  Theorem \ref{thm:main} can be generalized to incorporate distributions like the standard Cauchy distribution. This does not seem to be easy as \eqref{eq:suff_1} may flip for some negative $y$.
Indeed, from \eqref{eq:remark}, we can see that to let \eqref{eq:SD} hold for an arbitrary distribution $F$,  a sufficient and necessary condition is that
$$\int_{y=\essinf F}^x  \overline{F}\left(\frac{x-a y}{1-a}\right)\d F(y) + \int_{y=\essinf F}^x  \overline{F}\left(\frac{x-(1-a) y}{a}\right)\d F(y)$$
is increasing in $a\in[0,1/2]$ for all $x\in\R$. This condition can still hold if \eqref{eq:suff_1} does not hold for some values of $y$.
\end{remark}

Theorem \ref{thm:main} can be easily generalized to the case where the associated weight vectors of \eqref{eq:SD} are ordered in weak supermajorization order.
\begin{proposition}
Let $X_1,\dots,X_n\sim F$ be iid where $F\in\mathcal H$. For vectors $(\th_1, \ldots, \th_n)$ and $(\eta_1, \ldots, \eta_n)$ in $\R^n_+$ such that $(\th_1, \ldots, \th_n) \preceq^w (\eta_1, \ldots, \eta_n)$,
\begin{equation*}
\sum_{i=1}^n \eta_i X_i \leq_{\rm st} \sum_{i=1}^n \th_i X_i.
\end{equation*}
\end{proposition}
\begin{proof}
If $(\th_1, \ldots, \th_n) \preceq^w (\eta_1, \ldots, \eta_n)$, by Section 5.A.9a of \cite{MOA11}, there exists $(\lm_1, \ldots, \lm_n)$ such that
\begin{equation*}
(\th_1, \ldots, \th_n) \preceq (\lm_1, \ldots, \lm_n) ~~~ {\rm and} ~~~ (\lm_1, \ldots, \lm_n) \geq (\eta_1, \ldots, \eta_n).
\end{equation*}
Therefore, by Theorem \ref{thm:main}, for all $x>0$,
\begin{eqnarray*}
\P\(\sum_{i=1}^n \th_i X_i  > x\) \geq \P\(\sum_{i=1}^n \lm_i X_i > x\) \geq \P\(\sum_{i=1}^n \eta_i X_i > x\).
\end{eqnarray*}
The proof is  completed.
\end{proof}

A particularly interesting case of \eqref{eq:SD} relevant to risk sharing  is 
\begin{equation}\label{eq:SD*}
\(\sum_{i=1}^n \theta_i\) X_1 \le_{\rm st} \sum_{i=1}^n \theta_i X_i \mbox{~for all~} \overline\theta\in\R_+^n, \tag{\textcolor{red}{SD*}}
\end{equation}
where $X_1,\dots,X_n$ are iid random variables.  
Inequality \eqref{eq:SD*} implies that any rational decision maker will not share infinite-mean losses; we say a decision maker is rational if they prefer loss $X$ to loss $Y$ for $X\le_{\rm st}Y$.  \cite{CEW22} recently showed that \eqref{eq:SD*} holds if $X_1,\dots,X_n$ can be obtained via increasing and convex transformations of iid Pareto random variables with tail parameter 1; such random variables are called super-Pareto. Subsequently, this result was studied with broader classes of infinite-mean distributions by \cite{arab2024convex}, \cite{CS24}, and \cite{muller2024some}.

\begin{definition}
Let $F$ be a distribution function with $\essinf F=0$. We say $F$ is \emph{inverted subadditive}, denoted by  $F\in \H^*$, if
\begin{equation*} 
\overline{F} \left(\frac{1}{x_1}\right)+ \overline{F} \left(\frac{1}{x_2}\right)\ge  \overline{F} \left(\frac{1}{x_1+x_2}\right),~\mbox{for all}~(x_1,x_2)\in \R_{++}^2.\end{equation*}
\end{definition}
The class of inverted subadditive distributions, introduced by \cite{arab2024convex}, is clearly larger than the class of inverted concave distributions and contains the class of super-Pareto distributions. A random variable $X\in\H^*$ is called inverted subadditive because the inverted distribution
of $X$ is subadditive.  \cite{arab2024convex} established \eqref{eq:SD*} for inverted subadditive distributions \citep[][Theorem 3.1]{arab2024convex}.  This result is also implied by the last part of the proof of Theorem \ref{thm:main} (i.e., the inequality chain containing \eqref{eq:proofSD*}). 

\begin{theorem}\label{thm:main2}
If $F\in \H^*$, then $F$ satisfies \eqref{eq:SD*}.
\end{theorem}

We refer to \cite{arab2024convex}, \cite{CS24}, and \cite{muller2024some} for  detailed discussions of \eqref{eq:SD*}. Variations of \eqref{eq:SD*} were studied by \cite{CEW22} and \cite{CS24} for random variables that are possibly negatively dependent or non-identically distributed.

\subsection{Some closure properties of $\H$}\label{sec:closure}

 This section presents some closure properties of $\H$, which can be used to construct more distributions satisfying \eqref{eq:SD}.  We first define some useful notions of stochastic orders.
 \begin{definition}
For two random variables $X$ and $Y$,
 \begin{enumerate}[(i)]
 \item  
 $X$ is less than $Y$ in \emph{convex transform order}, denoted $X \leq_{\rm c} Y$, if $G^{-1}\(F(x)\)$ is convex in $x$, where $F$ and $G$ are the distribution functions of $X$ and $Y$, respectively;
     \item  $X$ is smaller than $Y$ in \emph{likelihood ratio order}, denoted by $X\le_{\rm lr}Y$, if $g(x)/f(x)$ is increasing in $x$, where $f$ and $g$ are the densities of $X$ and $Y$, respectively.
 \end{enumerate}
 \end{definition}

For continuous random variables $X$ and $Y$, if $Y\simeq_{\rm st}\phi(X)$ for some increasing and convex function $\phi$, then $X\le_{\rm c} Y$; see \cite{van1964convex} for further discussion of convex transform order. Let $U$ be uniformly distributed on $(0,1)$. Note that for continuous random variable $X$, $X\in\H$ if and only if $U\le_{\rm c} 1/X\simeq_{\rm st}\phi(U)$ where  $\phi$ is convex satisfying $\lim_{u\downarrow 0}\phi(u)=0$ and $\lim_{u\uparrow 1}\phi(u)=\infty$.
 
\begin{proposition}\label{prop:Hproperty}
Let $X\sim F$ where $F\in \H$. Then 
\begin{enumerate}[{\rm (i)}]
\item 
$F^\beta\in \H$ where $\beta\ge 1$;
\item
$\phi(X)\in\H$ where $\phi$ is strictly increasing, $\phi(0)=0$, and $1/\phi^{-1}(1/x)$ is concave in $x\in\R_+$;
\item 
if $Y\sim G$ is independent of $X$ where $G\in\H$, then $\max\{X,Y\}\in\H$;
\item 
if in addition $X$ and another  random variable $Y$ have densities and $X\le_{\rm lr}Y$, then $Y\in\H$.
\end{enumerate}
\end{proposition}
\begin{proof}
\begin{enumerate}[(i)]
\item 
Let $G=F^\beta$. Then $\overline G(1/x)=1-(1-\overline F(1/x))^\beta$, $x \in \R_{++}$. As $\beta\ge 1$, $1-(1-x)^\beta$ is increasing and concave in $x\in(0,1)$. Since $F\in\H$, $\overline G(1/x)$ is concave.
\item Note that $1/\phi(X)=g(1/X)$ where $g(x)=1/\phi(1/x)$, $x>0$. 
Since $1/\phi^{-1}(1/x)$ is strictly increasing and concave, $g$ is strictly increasing and convex. Thus $U\le_{\rm c}1/X\le_{\rm c}1/\phi(X)$ and $\phi(X)\in\H$.
\item
It suffices to show $U\le_{\rm c}1/\max(X,Y)$. Let strictly increasing and convex functions $f,g$ be such that $1/X\simeq_{\rm st}f(U)$ and $1/Y\simeq_{\rm st}g(U)$.  For $x>0$, $\p(1/\max(X,Y)\le x)=1-(1-f^{-1}(x))(1-g^{-1}(x))$. As $1-f^{-1}(x)$ and $1-g^{-1}(x)$ are positive, decreasing, and convex, the quantile function of $1/\max(X,Y)$ is convex. Thus $U\le_{\rm c}1/\max(X,Y)$.
\item
Let $Y$ have density $g$.
We have 
$x^2g(x)=x^2f(x)(g(x)/f(x))$. As $x^2f(x)$ and $g(x)/f(x)$ are both increasing in $x$ and are positive, $x^2g(x)$ is increasing in $x$. By Proposition \ref{prop:Hcont}, we have the desired result. \qedhere
\end{enumerate}
\end{proof}

Proposition \ref{prop:Hproperty} (ii) implies that some increasing and convex transforms on random variables preserve property $\H$. For instance, the conditions of Proposition \ref{prop:Hproperty} (ii) can be satisfied by commonly used convex functions $\phi(x)=x^p$ with $p\ge 1$ and $\phi(x)=e^x-1$ for $x>0$.  Intuitively, such convex transforms make random variables more ``heavy-tailed".
Direct calculations give the following result.
 \begin{proposition}\label{prop:mixture}
 The following statements hold.
 \begin{enumerate}[(i)]
 \item
  For $w_1,\dots,w_n\in \R_+$ such that $\sum_{i=1}^n w_i=1$, if distribution functions $F_1,\dots,F_n\in\H$, then
  $\sum_{i=1}^n w_iF_i\in\H$.  
  \item
  If $X\in \H$, then $\max(X-c,0)\in\H$ where $c\in \R_{++}$.
  \end{enumerate}
  \end{proposition}
 
  Proposition \ref{prop:mixture} (i) says that $\H$ is closed under distribution mixture. Proposition \ref{prop:mixture} (ii) is relevant in an insurance context: $X$ is the actual claim amount and $c$ is the deductible, i.e., $\max(X-c,0)$ is the payout by the insurer.

\section{Compound Poisson distributions}\label{sec:3}

In this section, we assume each random variable in \eqref{eq:SD} is a compound sum of random losses. 
More precisely, for each  $i\in[n]$, let   $X_i=\sum_{j=1}^{N_i}Y_{i,j}$, where $\{Y_{i,j}\}_{j\in\N}$ are iid random losses and $N_i$ is the number of losses,  which is independent of $\{Y_{i,j}\}_{j\in\N}$. The distribution of $X_i$ is referred to as a \emph{compound distribution}. 
One classic example of compound distributions is the compound Poisson distribution. 
\begin{definition}
Let random variables $Y_1,Y_2,\dots$ be iid with distribution $F$ and $N$ follows a Poisson distribution with mean $\lambda\in \R_{++}$, i.e., $\p(N=j)=\lambda^je^{-\lambda}/j!$, $j\in\N_0$. We say $X=\sum_{i=1}^NY_i$ follows a \emph{compound Poisson distribution} with Poisson parameter $\lambda$ and distribution $F$, denoted by ${\rm CP}(\lambda,F)$. 
\end{definition}

We first note that if a distribution $F$ satisfies \eqref{eq:SD}, \eqref{eq:SD} may not hold for ${\rm CP}(\lambda,F)$. For instance, if $F$ is a degenerate distribution at 1, then ${\rm CP}(\lambda,F)$ becomes a Poisson distribution, which clearly does not satisfy \eqref{eq:SD} by Proposition \ref{prop:trivial}. We give below a sufficient and necessary condition for \eqref{eq:SD} to hold for ${\rm CP}(\lambda,F)$.

\begin{theorem}\label{thm:CP}
Let $F$ be a continuous distribution function with $\essinf F=0$. Then ${\rm CP}(\lambda,F)$ satisfies  \eqref{eq:SD} for all $\lambda\in \R_{++}$ if and only if $F\in\H$.
\end{theorem}
\begin{proof}
We first show the ``$\Longleftarrow$" direction. Using an argument similar to the proof of Theorem \ref{thm:main}, it is sufficient to establish the result for the case of two random variables.  Let $X_1,X_2\sim {\rm CP}(\lambda,F)$ be independent and $\overline a=(a_1,a_2)\in(0,1)^2$ such that $a_1+a_2=1$. It is well known that  $\sum_{i=1}^2a_iX_i\sim {\rm CP}(2\lambda,H_{\overline a})$ where
$$H_{\overline a}(x)=\sum_{i=1}^2\frac{1}{2}F(x/a_i).$$
We give a proof for this result here for completeness. Let $Y\sim F$. The characteristic function of $\sum_{i=1}^2a_iX_i$  is 
\begin{align*}
\varphi_{\sum_{i=1}^2a_i X_i}(t) =\exp\left\{\lambda\(\sum_{i=1}^2\varphi_{Y}(a_it)-2\)\right\} =\exp\left\{2\lambda\(\frac{1}{2}\sum_{i=1}^2\varphi_{Y}(a_it)-1\)\right\}.
\end{align*}
As $\sum_{i=1}^2\varphi_{Y}(a_it)/2$ is the characteristic function of $H_{\overline a}$, $\sum_{i=1}^2a_iX_i\sim {\rm CP}(2\lambda,H_{\overline a})$.

Let $\overline \theta\preceq \overline \eta$ such that $\overline \theta, \overline \eta\in (0,1)^2$ and $\sum_{i=1}^2\theta_i = \sum_{i=1}^2\eta_i = 1$.  We need to show 
$$\sum_{i=1}^2\eta_iX_i\le_{\rm st}\sum_{i=1}^2\theta_iX_i.$$
Since $\sum_{i=1}^2\theta_iX_i\sim {\rm CP}(2\lambda,H_{\overline \theta})$ and $\sum_{i=1}^2\eta_iX_i\sim {\rm CP}(2\lambda,H_{\overline \eta})$,  it suffices to show 
$$\sum_{i=1}^NV_i\le_{\rm st}\sum_{i=1}^NU_i,$$ 
where
$N\sim {\rm Poi}(2\lambda)$,  $V_1,V_2,\dots\sim  H_{\overline \eta}$ and $U_1,U_2,\dots\sim  H_{\overline \theta}$ are mutually independent.  As $F\in\H$, by Proposition \ref{prop:Hcont}, we have $H_{\overline \theta}(x)\le H_{\overline \eta}(x)$ for all $x\in \R_+$, i.e., $V_1\le_{\rm st}U_1$. Moreover, as usual stochastic order is preserved under convolution,   for $x\in \R_+$, 
\begin{eqnarray*}
\p\(\sum_{i=1}^NV_i\le x\) &=& \sum_{m=1}^\infty \p\(\sum_{i=1}^mV_i\le x\)\p(N=m) \\ 
&\ge& \sum_{m=1}^\infty \p\(\sum_{i=1}^mU_i\le x\)\p(N=m)=\p\(\sum_{i=1}^NU_i\le x\).
\end{eqnarray*}
Hence, $\sum_{i=1}^2\eta_iX_i\le_{\rm st}\sum_{i=1}^2\theta_iX_i$ holds for $\overline \theta, \overline \eta\in (0,1)^2$. Moreover, if $\overline \eta=(0,1)$, $\sum_{i=1}^2\eta_iX_i\sim {\rm CP}(2\lambda,H^*_{\overline \eta})$ where 
$$H^*_{\overline \eta}(x)=\frac{1}{2}\(F(x)+\id_{\{x\ge 0\}}\).$$
Clearly, we have $H^*_{\overline \eta}(x)\ge H_{\overline \theta}(x)$ for all $x\in\R$. The  ``$\Longleftarrow$" direction is done. 

Next, we show the ``$\Longrightarrow$" direction. 
We have that, if $(\theta_1,\theta_2) \preceq (\eta_1,\eta_2)$ where $(\eta_1,\eta_2)\in(0,1)^2$,
 then for all $x\in \R$ and for some $\lambda\in\R_{++}$,
\begin{equation} \label{eq:aux_5}
\p(\theta_1 X_1 + \theta_2 X_2 > x) \ge \p(\eta_1 X_1 + \eta_2 X_2 > x),
\end{equation}
where $X_1,X_2\sim {\rm CP}(\lambda/2,F)$ are independent. Thus, by the previous arguments, \eqref{eq:aux_5} may be rewritten as
\begin{equation} \label{eq:aux_6}
 0 \le \lambda e^{-\lambda} [\p(U_1 > x) - \p(V_1 > x)] + \sum_{k = 2}^{\oo} \p(N=k) \[\p\left(\sum_{i=1}^k U_i > x\right) - \p\left(\sum_{i=1}^k V_i > x\right)\],
\end{equation}
where  $N\sim {\rm Poi}(\lambda)$, $V_1,V_2,\dots\sim  H_{\overline \eta}$ and $U_1,U_2,\dots\sim  H_{\overline \theta}$ are mutually independent.
Assume now that $F$ is not in $\H$, i.e., by Proposition \ref{prop:Hcont}, there exist vectors $(\theta_1,\theta_2) \preceq (\eta_1,\eta_2)$ and a value $x$ such that
$$
\[\overline{F}\left(\frac{x}{\theta_1}\right) + \overline{F}\left(\frac{x}{\theta_2}\right)\] - \[\overline{F}\left(\frac{x}{\eta_1}\right) + \overline{F}\left(\frac{x}{\eta_2}\right)\] = - \varepsilon,
$$
where $\varepsilon \in \R_{++}$. Let us now turn our attention back to \eqref{eq:aux_6}. The second term on the RHS may be bounded by
$$
\left|\sum_{k=2}^\oo \p(N=k) \[\p\left(\sum_{i=1}^k U_i > x\right) - \p\left(\sum_{i=1}^k V_i > x\right)\]\right| \le \p(N\ge 2) \le \lambda^2,
$$
while the first term is
$$
\lambda e^{-\lambda} (\p(U_1 > x) - \p(V_1 > x))  = - \frac{1}{2} \varepsilon \lambda e^{-\lambda} \le -\frac{1}{2} \varepsilon \lambda(1-\lambda).
$$
We have from \eqref{eq:aux_6} that, for all $\lambda\in \R_{++}$,
$$
0 \le -\frac12\varepsilon \lambda(1-\lambda) +  \lambda^2,
$$
which is clearly false for sufficiently small values of $\lambda$. Therefore, we have a contradiction and the proof is complete.
\end{proof}

\begin{remark}
By Theorem \ref{thm:CP}, if $F\in \H$,  CP$(\lambda, F)$ satisfies \eqref{eq:SD} for any $\lambda\in \R_{++}$. As usual stochastic order is preserved under mixtures  \citep[see, e.g.,][Theorem 1.A.3 (d)]{SS07}, if $F\in\H$, \eqref{eq:SD} holds for the compound mixed Poisson distribution CP$(\Lambda,F)$ where $\Lambda$ is a random variable.
\end{remark}

In the following, we give a sufficient and necessary condition for \eqref{eq:SD*} to hold for ${\rm CP}(\lambda,F)$.  As this result can be shown by a slight modification to the proof of Theorem \ref{thm:CP},   its proof is omitted.
\begin{theorem}
Let $F$ be a continuous distribution function with $\essinf F=0$. Then ${\rm CP}(\lambda,F)$ satisfies  \eqref{eq:SD*} for all $\lambda\in \R_{++}$ if and only if $F\in\H^*$.
\end{theorem}

\section{Stable distributions}\label{sec:4}

In this section, we present a discussion of \eqref{eq:SD} for stable random variables to complement the findings of \cite{ibragimov2005new}, who focused on the peakedness property of linear combinations of stable random variables. Denote by  $S_\alpha(\sigma,\beta,\mu)$  the stable distributions with stability parameter $\alpha \in (0,2]$, skewness parameter $\beta\in[-1,1]$, scale parameter $\sigma\in \R_{++}$, and shift parameter $\mu\in\R$. We say $S_\alpha(1,\beta,0)$ is a standard stable distribution.
The characteristic function of $ X\sim S_\alpha(\sigma,\beta,\mu)$ is given by,  for $x \in \R$,
\begin{equation*}
\E\big [\exp\{{\rm i} x X\}\big ] =\begin{cases}
\exp\left \{ {\rm i} \mu x-\sigma^\alpha |x|^\alpha \(1-{\rm i} \beta\sign(x)
\displaystyle  \tan \frac{\pi\alpha}{2}\)\right\}, & \mbox{if $\alpha\ne 1$},\\[8pt]
\exp\left\{{\rm i}\mu x-\sigma|x|\(1+{\rm i}\displaystyle\frac{2\beta}{\pi}\sign(x)\log|x|\)\right\},
&   \mbox{if $\alpha= 1$},
\end{cases}
\end{equation*}
where $\sign(\cdot) $ is the sign function. Stable random variables do not have finite mean if $\alpha\le 1$. Stable distributions with $\alpha\in(0,1)$ and $\beta=\pm1$ are one-sided, and the support  is $[\mu,\infty)$ for $\beta=1$ (resp.~$(-\infty,\mu]$ for $\beta=-1$). All other cases of stable distributions have both left and right tails. Stable distributions include normal distributions ($\alpha=2$), Cauchy distributions ($\alpha=1$ and $\beta=0$), Landau distributions ($\alpha=1$ and $\beta=1$), and L\'evy distributions ($\alpha=1/2$ and $\beta=1$).  
See \cite{No20} for further properties of stable distributions.

For non-degenerate iid random variables $X_1,\dots,X_n$, stable distributions are the limiting distributions of $a_n(X_1+\dots+X_n)+b_n$ for some $a_n\in \R_{++}$ and $b_n\in\R$. If $\E[X_1^2]<\infty$, the limiting distribution is the normal distribution, otherwise the limiting distribution is the stable distribution with $\alpha<2$. By Proposition \ref{prop:trivial}, we have the following result.
\begin{proposition}
If $X_1,\dots,X_n$ are iid and satisfy \eqref{eq:SD} and  $a_n(X_1+\dots+X_n)+b_n\xrightarrow{d}Z$ for some stable random variable $Z$, where $a_n\in \R_{++}$ and $b_n\in\R$, then $Z$ satisfies \eqref{eq:SD}.
\end{proposition}

Indeed, one can use the nice properties of stable distributions to give the following results of \eqref{eq:SD} for stable distributions. We assume that the stable distributions in the following result are standard for simplicity, since usual stochastic order is preserved under location-scale transforms.
Theorem \ref{thm:stable} (ii) below has appeared in Theorem 1.2.4 of \cite{ibragimov2005new} in a different form; we provide a proof here for completeness.

\begin{theorem}\label{thm:stable}
The stable distribution $S_\alpha(1,\beta,0)$ satisfies \eqref{eq:SD} if and only if one of the following conditions holds.
\begin{enumerate}[{\rm (i)}]
\item $\alpha=1$ and $\beta\in \R_+$;
\item  $\alpha< 1$ and $\beta=1$.
\end{enumerate}
\end{theorem}
\begin{proof}
Let $X_1,\dots,X_n\sim S_\alpha(1,\beta,0)$ be independent and $(w_1,\dots,w_n)\in\R_+^n$ such that $\sum_{i=1}^nw_i=1$. We first show ``$\Longleftarrow$'' direction.  If $\alpha=1$, 
it is seen from (1.7) of \cite{No20} that
\begin{equation}\label{eq:stable1}
\sum^n_{i=1} w_i X_i \simeq_{\rm st} X_1-\frac{2\beta}{\pi}\sum_{i=1}^nw_i\log w_i.
\end{equation}
As $g(t) =-t\log(t)$ is concave in $t\in (0,1)$, $\sum_{i=1}^ng(t_i)$ is Schur-concave in $(t_1,\dots, t_n)\in \R_+^n$ \citep[see][Proposition 3.C.1]{MOA11}. Hence, we have (i). 

Similarly, to show (ii), we have from (1.7) of \cite{No20} that if $\alpha\neq 1$,
\begin{equation}\label{eq:stable2}
\sum^n_{i=1} w_i X_i \simeq_{\rm st} \(\sum_{i=1}^nw_i^\alpha\)^{1/\alpha}X_1.
\end{equation}
If $\alpha< 1$ and $\beta=1$,  $X_1$ is positive almost surely and $\sum^n_{i=1} c_i^\alpha$ is Schur-concave in $(c_1, \dots, c_n) \in \R^n_+$. Thus, we have the desired result.

For the ``$\Longrightarrow$'' direction, we first note that if $S_\alpha(1,\beta,0)$ satisfies \eqref{eq:SD}, then $\alpha\le 1$; this is due to Proposition \ref{prop:trivial} (i).  By \eqref{eq:stable1}, if $\alpha=1$, we must have $\beta\in \R_+$. If $\alpha<1$, assume that $\beta<1$. Then for $u\le0$,  by \eqref{eq:stable2}, we have 
 \begin{align*}
 \P\(\sum^n_{i=1} w_i X_i \le u\)= \P\(X_1 \le  \frac{u}{(\sum^n_{i=1} w_i^\alpha)^{1/\alpha}}\).
 \end{align*}
 Thus $\P\(\sum^n_{i=1} w_i X_i \le u\)$ is Schur-concave in $(w_1,\dots,w_n)$ and we have a contradiction. Therefore, if $\alpha<1$,  $\beta=1$.
\end{proof}
Theorem \ref{thm:stable} (i) suggests that for a random variable to satisfy \eqref{eq:SD}, the random variable is not necessarily bounded from below, which is a key assumption of Theorem \ref{thm:main}; see also Remark \ref{rem:unbounded}.

As the absolute Cauchy distribution satisfies \eqref{eq:SD}, a natural question is whether absolute stable distributions with parameter $\alpha\le 1$ also satisfy \eqref{eq:SD}. We conjecture that the answer is yes but do not have a proof. We use the stochastic representation of stable random variables in \cite{WR96} to simulate 10000 random samples of two independent stable random variables $X_1,X_2\sim S_\alpha(1,\beta,0)$; see Appendix \ref{sec:B} for the stochastic representation. Let  $Y_1= a_1|X_1|+a_2|X_2|$ and $Y_2= b_1|X_1|+b_2|X_2|$. We consider three cases of the weight vector $(a_1, a_2, b_1, b_2) = (1, 4, 2, 3)$, $(2, 6, 3, 5)$ or $(3, 8, 5, 6)$.
Figure \ref{Stable-1} plots the empirical distribution functions of $Y_1$ and $Y_2$ with $\alpha=0.10$ and $\beta=0$ or $-0.40$, respectively.
Figure \ref{Stable-2} plots the empirical distribution functions of $Y_1$ and $Y_2$ with $\alpha=0.90$ and $\beta=0$ or $0.30$, respectively. From Figures \ref{Stable-1} and \ref{Stable-2}, we find that the empirical distributions $\widehat{F}_1$ and $\widehat{F}_2$ of $Y_1$ and $Y_2$ seem to satisfy $\widehat{F}_1(x) \ge \widehat{F}_2(x)$ for all $x$. 
\begin{figure}[h]
\centering
\includegraphics[width=7cm]{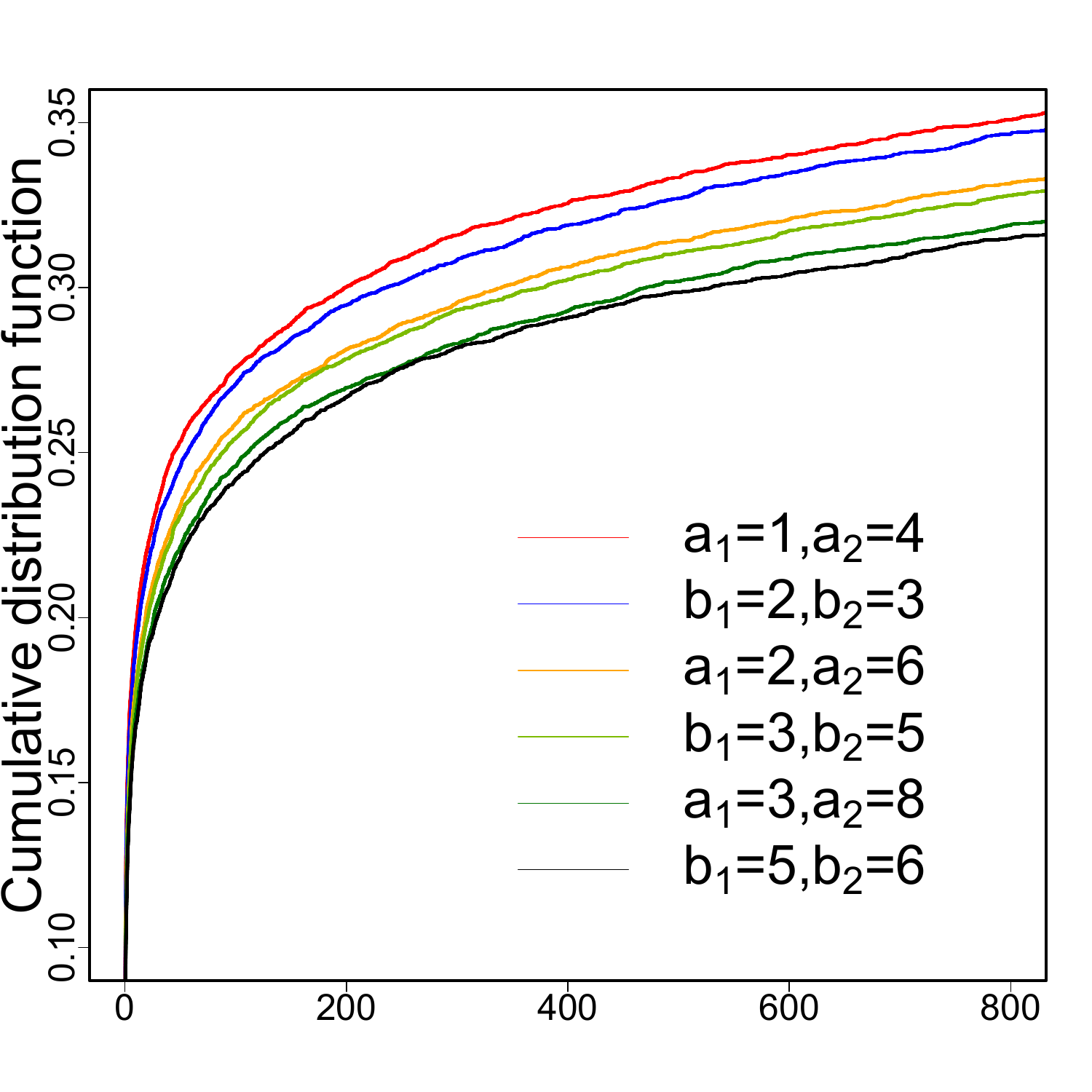}
\centering
\includegraphics[width=7cm]{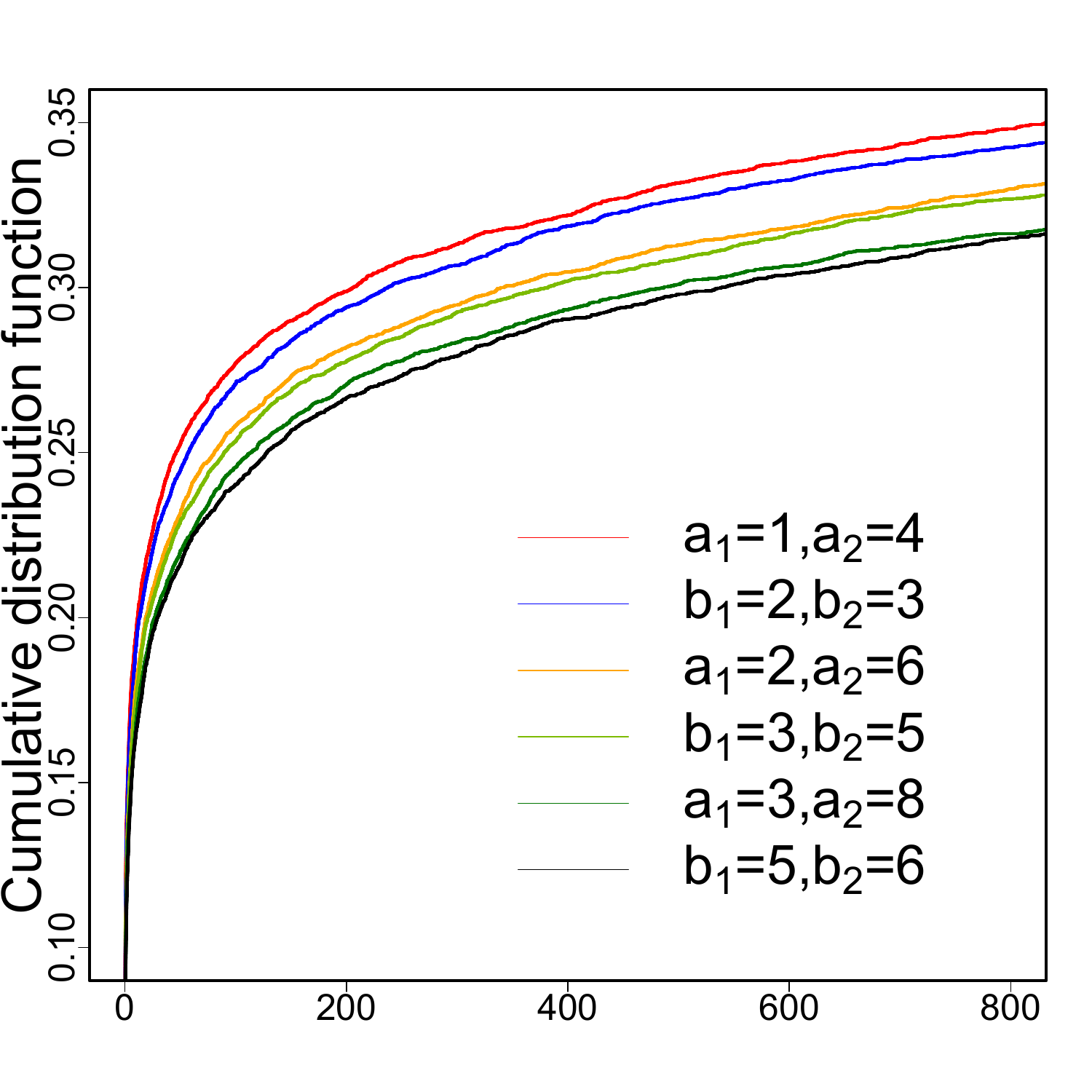}
\caption{Empirical distribution functions of $Y_1$ and $Y_2$ with $\alpha=0.1,\ \beta=0$ (left) and $\alpha=0.1,\ \beta=-0.4$ (right).}    \label{Stable-1}
\end{figure}
\begin{figure}[h]
\centering
\includegraphics[width=7cm]{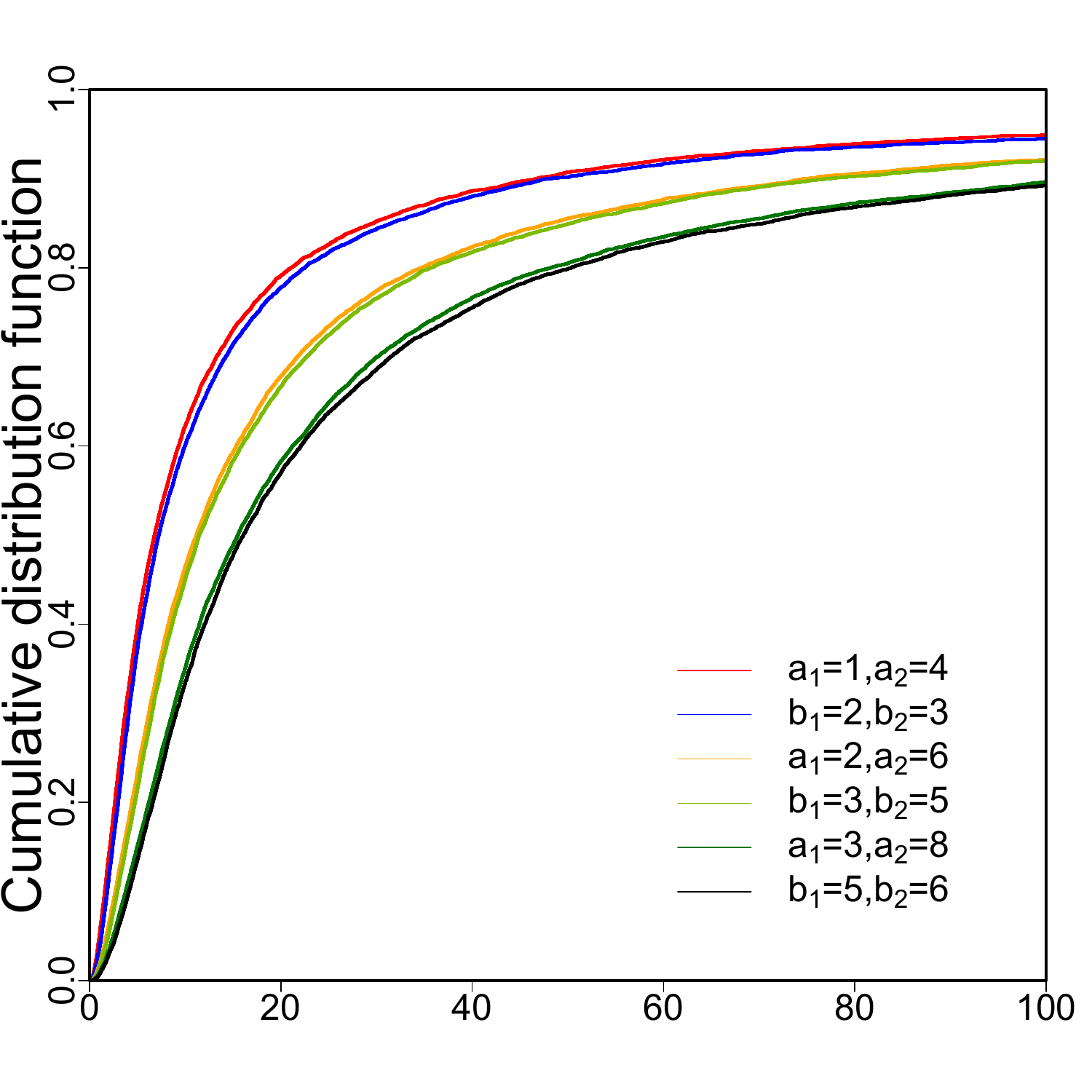}
\centering
\includegraphics[width=7cm]{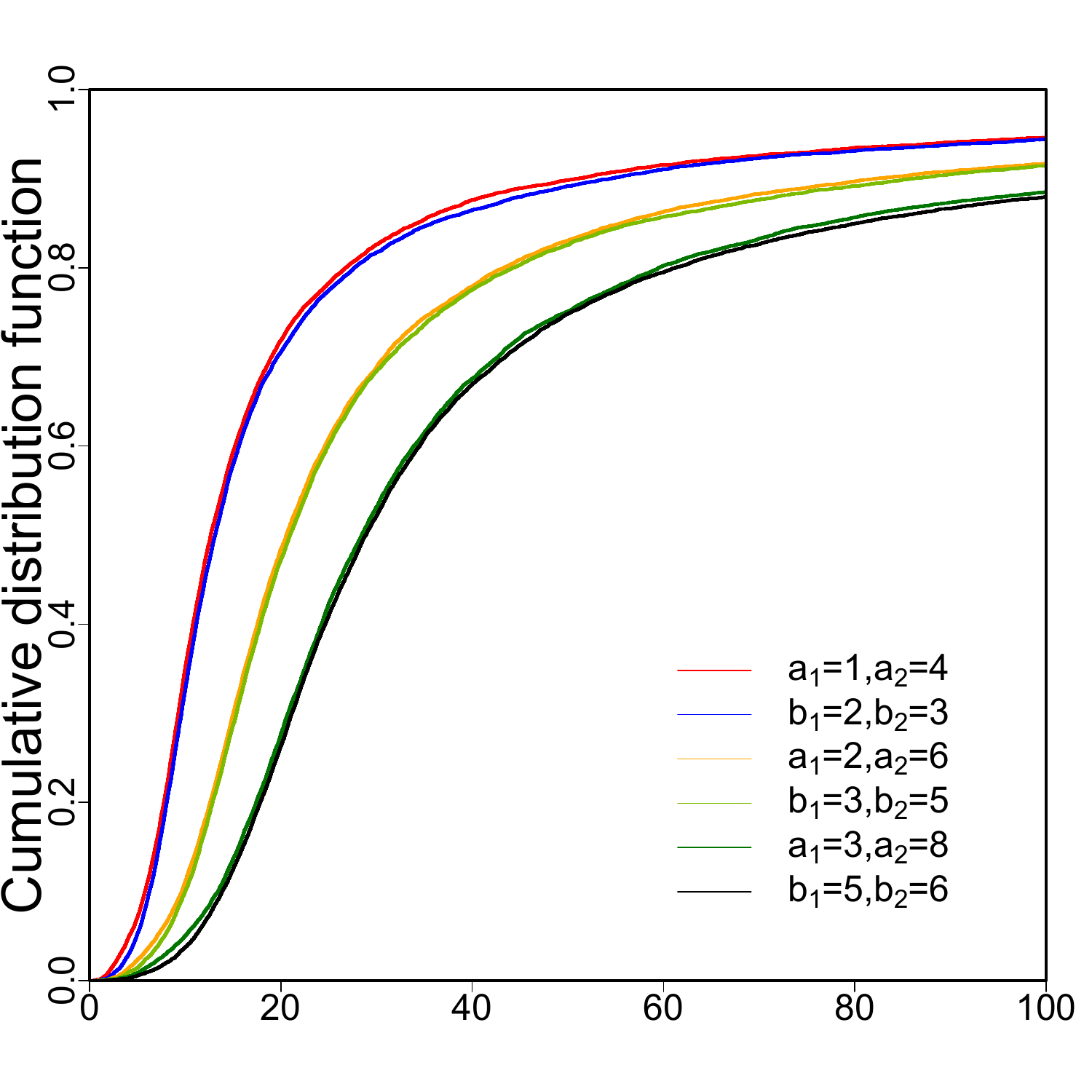}
\caption{Empirical distribution functions of $Y_1$ and $Y_2$ with $\alpha=0.9,\ \beta=0$ (left) and $\alpha=0.9,\ \beta=0.3$ (right).}
\label{Stable-2}
\end{figure}

\section{Concluding remarks}\label{sec:5}

In this paper, we study sufficient and necessary conditions for the stochastic dominance relation \eqref{eq:SD} across different classes of distributions. While previous studies have established \eqref{eq:SD} for a limited set of distributions (stable distributions with tail parameter strictly less than 1 and Pareto distributions with infinite mean), we broaden the scope by proving that \eqref{eq:SD} holds for a much richer class of distributions, called the inverted concave distribution.
 Moreover, if each random variable in \eqref{eq:SD} is a compound Poisson sum, then \eqref{eq:SD} holds if and only if the summand of the compound Poisson sum is inverted concave. We conclude with some open questions.

Theorem \ref{thm:main} provides a sufficient condition for \eqref{eq:SD} with the assumption that random variables are bounded from below. How far is this condition from being necessary? Moreover, as \eqref{eq:SD} can also hold for random variables taking values on the whole real line, such as the Cauchy distribution,  what is the sufficient condition of \eqref{eq:SD} without the assumption that random variables are bounded from below?

Examples \ref{ex:absCauchy} and \ref{ex:IG} show that the absolute Cauchy and the L{\'e}vy distributions are inverted concave. A natural question is whether absolute stable random variables with tail parameters less than or equal to 1 and positive stable random variables are also inverted concave.
 The absence of closed-form expressions for distribution functions in most stable distributions creates a methodological challenge in directly verifying conditions for inverted concave distributions, necessitating alternative analytical techniques.

Another relevant question is whether all absolute stable random variables with tail parameters less than or equal to $1$ satisfy property \eqref{eq:SD}. The answer seems to be yes based on some simulation results.

It was shown by \cite{CHWZ25} that \eqref{eq:SD} holds for infinite-mean Pareto random variables that are positively dependent via some Clayton copula. The last question is whether such a result can be generalized for more general classes of distributions and for more general dependence structures of random variables.

\newpage

{

\appendix

\setcounter{table}{0}
\setcounter{figure}{0}
\setcounter{equation}{0}
\renewcommand{\thetable}{A.\arabic{table}}
\renewcommand{\thefigure}{A.\arabic{figure}}
\renewcommand{\theequation}{A.\arabic{equation}}

\setcounter{theorem}{0}
\setcounter{proposition}{0}
\renewcommand{\thetheorem}{A.\arabic{theorem}}
\renewcommand{\theproposition}{A.\arabic{proposition}}
\setcounter{lemma}{0}
\renewcommand{\thelemma}{A.\arabic{lemma}}

\setcounter{example}{0}
\renewcommand{\theexample}{A.\arabic{example}}

\setcounter{corollary}{0}
\renewcommand{\thecorollary}{A.\arabic{corollary}}

\setcounter{remark}{0}
\renewcommand{\theremark}{A.\arabic{remark}}
\setcounter{definition}{0}
\renewcommand{\thedefinition}{A.\arabic{definition}}

\begin{center}
\LARGE	Appendices
\end{center}
\section{Examples of inverted concave distributions}\label{sec:SHex}

In this section, we show that many infinite-mean distributions are in $\H$. For a distribution function $F$, we use $f$ to denote its density and let $h(x)=x^2f(x)$, $x\in \R_+$.

\begin{example}[Pareto distribution]\label{ex:Pareto}
For $\alpha\in \R_{++}$, the Pareto distribution function is
$$F(x)=1-\frac{1}{(x+1)^\alpha},~~ x\in \R_{++}.$$
If $\alpha\le 1$, then $F$ has infinite mean.
Let $g(x)=\overline F(1/x)=(x/(x+1))^\alpha$, $x\in \R_+$.  Taking the second-order derivative of $g$, we have 
$$g''(x)=\frac{\alpha (\alpha-2x-1)}{x^{2-\alpha}(x+1)^{2+\alpha}}.$$ 
If $\alpha\le 1$, then $g''$ is negative.  Thus, by Proposition \ref{prop:Hcont}, we have $F\in \H$ if $\alpha\le 1$.
\end{example}

\begin{example}[Log-Pareto distribution]
The log-Pareto distribution  \citep[see][p.~39]{A15} has distribution function 
$$F(x)=1-\frac{1}{(\log(x+1)+1)^\alpha},~~x\in \R_{++}.$$ 
The derivative of $h$ is 
$$h'(x)=-\frac{\alpha x}{(x+1)^2 (\log(x+1)+1)^{\alpha+2}} \big[\alpha x-(x+2)\log(x+1)-2\big].$$
Let 
$$g(x)=\alpha x-(x+2)\log(x+1)-2,~~x\in \R_+.$$
By Proposition \ref{prop:Hcont}, we need $h'(x)\ge 0$, or equivalently, $g(x)\le 0$, such that  $F\in\H$. As $g(0)=-2$, it is sufficient to choose $\alpha$ such that $g$ is decreasing. We have 
$$g'(x)=\alpha-\frac{x+2}{x+1}-\log(x+1).$$
Thus if $\alpha\le 2$, then $F\in\H$.
\end{example}

\begin{example}[Inverse Burr distribution]
The inverse Burr distribution is defined as 
\begin{equation*}
F(x)=\(\frac{x^\tau}{x^\tau+1}\)^\alpha,~~ x\in \R_{++},
\end{equation*}
where $\alpha,\tau\in \R_{++}$.  For $x\in \R_+$,
$$h'(x)=\alpha\tau x^{\alpha\tau}(x^\tau+1)^{-\alpha-2}(\alpha\tau-(\tau-1)x^\tau+1).$$ 
Hence, if $\tau\le 1$, then $h'(x) \geq 0$. By Proposition \ref{prop:Hcont}, we have $F\in \H$.

\end{example}

\begin{example}[Stoppa distribution]
For $\alpha, \beta\in \R_{++}$, a (location-shifted) Stoppa distribution can be defined as
\begin{equation*}
F(x)=\(1-\frac{1}{(x+1)^\alpha}\)^\beta, ~~x\in \R_{++}.
\end{equation*}
Since a Stoppa distribution is a power transform of a Pareto distribution,
by Proposition \ref{prop:Hproperty} (i), if $\alpha\le 1$ and $\beta\ge 1$, then $F\in\H$.  Power transforms have also been used to generalize Burr distributions \citep[see][p.~211]{KK03}.
\end{example}

\begin{example}[Log-Cauchy distribution]
The log-Cauchy distribution function is 
$$F(x)=\frac{1}{\pi}\arctan(\log(x))+\frac{1}{2},~~x\in \R_{++}.$$
We have
$$f(x)=\frac{1}{x^2\pi}\frac{x}{(\log (x))^2+1}.$$
As $x/(\log(x)^2+1)$ is increasing in $x$, then $F\in\H$ by Remark \ref{rem:RV}.
\end{example}

\section{Simulation of stable distributions}\label{sec:B}
We recall below the simulation procedure of standard stable distributions $S_\alpha(1,\beta,0)$ in  \cite{WR96}. 
\begin{itemize}
\item[(1)] Let two  random variables  $V \sim U(-\pi/2, \pi/2)$ and $W \sim {\rm Exp}(1)$ be independent.
\item[(2)] For $\alpha \neq 1$, let 
$$ 
X = M_{\alpha, \beta} \frac{\sin(\alpha(V-N_{\alpha, \beta}))}{(\cos(V))^{1/\alpha}}
\left(\frac{1}{W} \cos\Big (V-\alpha(V-N_{\alpha, \beta})\Big ) \right)^{(1-\alpha)/\alpha},
$$
where
$$
M_{\alpha, \beta} = \left[1 + \beta^2 \tan^2\left(\frac{\pi\alpha}{2}\right)\right]^{1/(2\alpha)}, \qquad  N_{\alpha, \beta} = - \frac{1}{\alpha} \arctan\left(\beta \tan \frac{\pi\alpha}{2}\right).
$$

\item[(3)] For $\alpha = 1$, let
$$
X=\frac{2}{\pi}\left[
\left(\frac{\pi}{2}+\beta V\right)\tan(V)
-\beta\log\left(
\frac{(\pi/2)W\cos(V)}{\pi/2+\beta V}
\right)\right].
$$
\end{itemize}
As $X \sim S_\alpha(1, \beta, 0)$, one can generate samples of $X$ by first generating samples of $V$ and $W$ and then plug in the samples to the above  stochastic representations of $X$.

}

\section*{Acknowledgments}
The authors would like to thank an associate editor and two anonymous referees for their helpful comments on an earlier version of this paper.
Y.~Chen is supported by the 2025 Early Career Researcher Grant from the University of Melbourne.
T.~Hu would like to acknowledge financial support from the National Natural Science Foundation of China (No.~72332007, 12371476).
Z.~Zou is supported by the National Natural Science Foundation of China (No. 12401625), the China Postdoctoral Science Foundation (No. 2024M753074), and the Postdoctoral Fellowship Program of CPSF (GZC20232556).

\section*{Disclosure statement}

No potential conflict of interest was reported by the authors.

\end{document}